\documentclass[11pt,a4paper]{article}

\usepackage{amsmath,amsfonts,amssymb,latexsym,graphics,epsfig,url}
\usepackage{hyperref}
\usepackage{color}
\usepackage{amsthm}
\usepackage{tikz}
\usetikzlibrary{arrows}

\usepackage[english]{babel}
\usepackage{graphicx}

\newcommand{\executeiffilenewer}[3]{%
\ifnum\pdfstrcmp{\pdffilemoddate{#1}}%
{\pdffilemoddate{#2}}>0%
{\immediate\write18{#3}}\fi%
}

\theoremstyle{plain}
\newtheorem{theorem}{Theorem}[section]
\newtheorem{lemma}[theorem]{Lemma}

\newtheorem{proposition}[theorem]{Proposition}

\newtheorem{definition}[theorem]{Definition}
\newtheorem{quest}[theorem]{Question}

\DeclareMathOperator{\Cay}{Cay}
\DeclareMathOperator{\diag}{diag}
\DeclareMathOperator{\dist}{dist}

\DeclareMathOperator{\End}{End}

\DeclareMathOperator{\tr}{tr}
\DeclareMathOperator{\som}{sum}
\DeclareMathOperator{\spec}{sp}
\def\spaan{\mathop{\rm span }\nolimits}

\def\e{\mbox{\boldmath $e$}}
\def\j{\mbox{\boldmath $j$}}

\def\x{\mbox{\boldmath $x$}}

\def\alpha{\mbox{\boldmath $\alpha$}}
\def\nu{\mbox{\boldmath $\nu$}}
\def\rho{\mbox{\boldmath $\rho$}}
\def\xi{\mbox{\boldmath $\xi$}}
\def\0{\mbox{\boldmath $0$}}

\def\00{\mbox{\boldmath $O$}}

\def\A{\mbox{\boldmath $A$}}

\def\B{\mbox{\boldmath $B$}}

\def\D{\mbox{\boldmath $D$}}

\def\I{\mbox{\boldmath $I$}}

\def\M{\mbox{\boldmath $M$}}
\def\N{\mbox{\boldmath $N$}}

\def\S{\mbox{\boldmath $S$}}
\def\T{\mbox{\boldmath $T$}}
\def\X{\mbox{\boldmath $X$}}
\def\Y{\mbox{\boldmath $Y$}}

\def\G{\Gamma}
\def\Re{\mathbb R}
\def\Z{\mathbb Z}

\begin{document}
\title{Distance mean-regular graphs}
 \author{V. Diego$^a$ and M.A. Fiol$^{a,b}$
\\ \\
{\small $^a$Universitat Polit\`ecnica de Catalunya} \\
{\small Dept. de Matem\`atica Aplicada IV, Barcelona, Catalonia}\\
{\small $^b$Barcelona Graduate School of Mathematics}\\
{\small E-mails: {\tt \{victor.diego,fiol\}@ma4.upc.edu}} \\
 }

\maketitle

\begin{abstract}
We introduce the concept of distance mean-regular graph, which can be seen as a generalization of both vertex-transitive and distance-regular graphs. Let $\G$ be a graph with vertex set $V$, diameter $D$, adjacency matrix $\A$, and adjacency algebra ${\cal A}$. Then, $\G$ is {\em distance mean-regular} when, for a given $u\in V$, the averages of the intersection numbers
$p_{ij}^h(u,v)=|\G_i(u)\cap \G_j(v)|$ (number of vertices at distance $i$ from $u$ and distance $j$ from $v$)
computed over all vertices $v$ at a given distance $h\in \{0,1,\ldots,D\}$ from $u$, do not depend on $u$. In this work we study some properties and characterizations of these graphs. For instance, it is shown that a distance mean-regular graph is always distance degree-regular, and we give a condition for the converse to be also true.
Some algebraic and spectral properties of distance mean-regular graphs are also investigated. We show that, for distance mean regular-graphs, the role of the distance matrices of distance-regular graphs is played for the so-called  distance mean-regular matrices. These matrices are computed from a sequence of orthogonal polynomials evaluated at the adjacency matrix of $\G$ and, hence, they generate a subalgebra of ${\cal A}$. Some other algebras associated to distance mean-regular graphs are also characterized.
\end{abstract}


\section{Preliminaries}
Distance-regular graphs are a key concept in combinatorics, because of their rich structure and multiple applications. Indeed, they have important connections with other branches of mathematics, as well as other areas of graph theory.
For background on distance-regular graphs, we refer the reader to Bannai and Ito \cite{bi93}, Bigss \cite{b93}, Brouwer, Cohen, and Neumaier \cite{bcn89}, Brouwer and Haemers \cite{bh12},  Godsil \cite{g93}, and Van Damm, Koolen and Tanaka \cite{dkt12}.


Since their introduction by Biggs \cite{b71} in the early 70's, most of generalizations proposed of distance-regular graphs  are basically intended for regular graphs. For instance, Weichel \cite{w82} introduced the so-called {\em distance-polynomial graphs}, as those having their distance matrices in the adjacency algebra of the graph. Another example is the {\em distance-degree regular} or {\em super-regular} graphs, proposed by Hilano and Nomura \cite{hn84}, and characterized by the independence of the number of vertices at a given distance from every vertex.


\subsection{Graphs and their spectra}
Before discussing our approach and related results, we begin by recalling some basic notation and  background.
Let $\G=(V,E)$ be a finite, simple, and connected graph
with vertex set $V$, order $n=|V|$,
and diameter $D$. The distance between vertices $u$ and $v$ is denoted by $\dist(u,v)$. The set
of vertices at distance $i$ from a given vertex $u\in V$ is $\G_i(u)=\{v:\dist(u,v)=i\}$, for $i=0,\ldots,D$. If the cardinality $k_i(u)=|\G_i(u)|$ does not depend on $u$, we write it by $k_i$. Thus, $\G$ is {\em super-regular} \cite{hn84} if this happens for every $i=0,\ldots,D$. In particular, if the graph is regular, we denote its degree by $k=k_1$.
The {\em distance-$i$ graph} $\G_i$ has the same vertex set as $\G$, and two vertices are adjacent if and only if they are at distance $i$ in $\G$.
If $\G$ has adjacency matrix $\A$,  we denote
the spectrum of $\G$  by
$$
\spec \G = \spec \A = \{\lambda_0^{m_0},\ldots,\lambda_d^{m_d}\},
$$
where the different eigenvalues of $\G$ are in decreasing order,
$\lambda_0>\cdots >\lambda_d$, and the superscripts
stand for their multiplicities $m_i=m(\lambda_i)$.
The adjacency matrix of $\G_i$ is referred to as the {distance-$i$ matrix} $\A_i$. Thus, $\A_0=\I$ and $\A_1=\A$.

\subsection{Interlacing and equitable partitions}
\label{interlacing}
The following basic results about interlacing and equitable partitions can be found in
Haemers \cite{h95}, Fiol \cite{f99}, or Brouwer and Haemers \cite{bh12}.
Given a graph $\G$ on $n$ vertices, and with eigenvalues $\theta_1\ge \theta_2\ge \cdots \ge \theta_n$,
let ${\cal P}=\{U_1,\ldots,U_m\}$ be a partition of its vertex set $V$.  Let $\T$ be the
characteristic matrix of $\cal P$, which columns are
the characteristic vectors of $U_1,\ldots,U_m$, and consider the matrix $\S=\T\D^{-1}$ where
$\D=\diag(|U_1|,\ldots,|U_m|)=\T^\top\T$, satisfying $\S^{\top}\T=\I$. Then the so-called
{\em quotient matrix} of $\A$ with respect to ${\cal P}$, is
\begin{equation}
\label{quotient-matrix}
\B=\S^{\top}\A\T=\D^{-1}\T^{\top}\A\T,
\end{equation}
and its element $b_{ij}$ equals the average row
sum of the block $\A_{i,j}$ of $\A$ with rows and columns
indexed by $U_i$ and $U_j$, respectively.
Moreover, the  eigenvalues $\mu_1\ge \mu_2\ge \cdots\ge \mu_m$ of $\B$ interlace those of $\A$, that is,
\begin{equation}
\label{ineq:interlacing}
\theta_i\ge \mu_i\ge \theta_{n-m+i},\qquad i=1,\ldots, m.
\end{equation}
Moreover, if the interlacing is tight, that is, there exists some $k$, $0\le k\le m$, such that
$\theta_i=\mu_i$, $i=1,\ldots,k$, and $\mu_i=\theta_{n-m+i}$, $i=k+1,\ldots, m$, then ${\cal P}$ is an {\em equitable} (or {\em regular}) partition of $\A$, that is, each
block of the partition has constant row and column sums. In
the graph $\G$, this means that the bipartite induced subgraph
$\G[U_i,U_j]$ is biregular for each $i\neq j$, and that the
induced subgraph $\G[U_i]$ is regular for each
$i\in\{1,\ldots,m\}$.

\section{Distance mean-regular graphs}
In this section we introduce the concept of distance mean-regular graph and give some of their basic properties.
The reader will realize soon that most of such properties are similar to those of distance-regular graphs.

\subsection{Definitions and examples}
First, recall that a graph $\G=(V,E)$ with diameter $D$ is {\it distance-regular} if, for any two vertices $u,v\in V$ at distance $h=\dist(u,v)$, the numbers ({\em intersection parameters})
$$
p_{ij}^h(u,v)=|\G_i(u)\cap \G_j(v)|,\qquad h,i,j=0,\ldots,D,
$$
only depend on $h,i,j$.

Inspired by this definition, we consider the following generalization:

\begin{definition}
Given a graph $\G$ with a vertex $u\in V$ we consider the averages
$$
\overline{p}_{ij}^h(u)=\frac{1}{|\G_h(u)|}\sum_{v\in \G_h(u)}|\G_i(u)\cap \G_j(v)|,\qquad h,i,j=0,\ldots,D.
$$
If these  numbers do not depend on the vertex $u$, but only on the integers $h,i,j$, we say that $\G$ is {\em distance mean-regular} with parameters
$\overline{p}_{ij}^h$.
\end{definition}

As in the case of distance-regular graphs,  when $i=1$ we use the abbreviated   notations
$\overline{a}_h=\overline{p}_{1h}^h$, $\overline{b}_h=\overline{p}_{1,h+1}^h$, and $\overline{c}_h=\overline{p}_{1,h-1}^h$. Moreover, in due course we shall show that, under a general condition, the invariance of these {\em intersection numbers} suffices for having distance mean-regularity.

The {\em intersection mean-matrix} $\overline{\B}$ of $\G$, with entries $(\overline{\B})_{hj}=\overline{p}_{1j}^h$, has the tridiagonal form
$$
\overline{\B}= \left(\begin{array}
{cccccc}
\overline{a}_0 & \overline{b}_0   &                &                    &                    &                \\
\overline{c}_1 & \overline{a}_1   & \overline{b}_1 &                    &                    &                \\
               & \overline{c}_2   & \overline{a}_2 & \overline{b}_2     &                    &                \\
               &                  &                &   \ddots           &                    &               \\
               &                  &                & \overline{c}_{d-1} & \overline{a}_{d-1} & \overline{b}_{d-1}\\
               &                  &                &                    & \overline{c}_d     &  \overline{a}_d
\end{array}\right),
$$
(notice that $\overline{a}_0=0$, and $\overline{p}_{1j}^h=0$ for $j\neq h-1,h,h+1$), with corresponding intersection mean-diagram of Fig.\ref{intersec-mean-diagram} (where $\omega_{ij}$ denotes the number of edges between vertex sets $\G_i(u)$ and $\G_j(u)$).
As the row sums of $\overline{\B}$ are equal to the degree of $\G$ (see Lemma \ref{basic-properties}$(i)$), the same information can be represented by the {\em intersection mean-array}, which is
$$
\iota(\G)=\{
\overline{b}_0, \overline{b}_1,\ldots,\overline{b}_{d-1}; \overline{c}_1,\overline{c}_2,\ldots,\overline{c}_d
\}.
$$


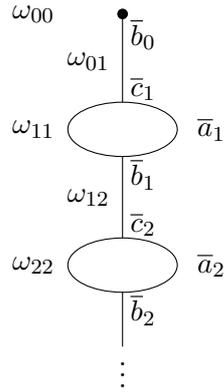
\begin{figure}[h!]
\begin{center}
\begin{tikzpicture}[scale=0.9]

\node  at (-1.3,-0.3) {$\omega_{00}$};
\draw[fill] (0,-0.3) circle (2pt);
\node  at (0.3,-0.6) {$\overline{b}_0$};

\draw (0,-0.3) --(0,-1.6);
\node at (-0.5,-1) {$\omega_{01}$};

\node  at (0.3,-1.4) {$\overline{c}_1$};
\node  at (-1.3,-2) {$\omega_{11}$};
\draw[rotate=0] (0,-2) ellipse (0.8 and 0.4);
\node  at (1.3,-2) {$\overline{a}_1$};
\node  at (0.3,-2.66) {$\overline{b}_1$};

\draw (0,-2.4) --(0,-3.6);
\node at (-0.5,-3) {$\omega_{12}$};

\node  at (0.3,-3.4) {$\overline{c}_2$};
\node  at (-1.3,-4) {$\omega_{22}$};
\draw[rotate=0] (0,-4) ellipse (0.8 and 0.4);
\node  at (1.3,-4) {$\overline{a}_2$};
\node  at (0.3,-4.66) {$\overline{b}_2$};

\draw (0,-4.4) --(0,-5.2);
\node  at (0,-5.5) {$\vdots$};

\end{tikzpicture}
\end{center}
\vskip -.8cm
\caption{Mean-intersection diagram.}
\label{intersec-mean-diagram}
\end{figure}

%

Some instances of distance mean-regular graphs are the
distance-regular graphs, and the vertex-transitive graphs.
Thus, a first example of distance mean-regular (vertex-transitive) graph is the circulant (or Cayley graph) $\G=\Cay(\Z_8;\pm 4)$ shown in Fig. \ref{fig1}.
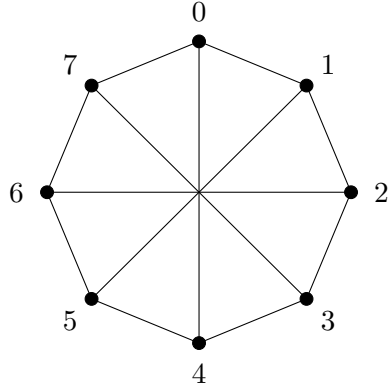
\begin{figure}[h]
\begin{center}
\begin{tikzpicture}[scale=0.4]
\foreach \i in {0,...,7}
{
\path (45*\i:5) coordinate (P\i);
\path (45*\i:6) coordinate (Q\i);
\draw[fill] (P\i) circle (6pt);
}
\node  at (Q0) {$2$};
\node  at (Q1) {$1$};
\node  at (Q2) {$0$};
\node  at (Q3) {$7$};
\node  at (Q4) {$6$};
\node  at (Q5) {$5$};
\node  at (Q6) {$4$};
\node  at (Q7) {$3$};
\foreach \i / \j  in {0/1,1/2,2/3,3/4,4/5,5/6,6/7,7/0}
{
 \draw (P\i) --(P\j);
}
\foreach \i / \j  in {0/4,1/5,2/6,3/7}
{
 \draw (P\i) --(P\j);
}
\end{tikzpicture}
\end{center}
\caption{A vertex-transitive distance mean-regular graph.}
\label{fig1}
\end{figure}
This graph has
diameter $D=2$, and intersection mean-matrix
$$
\overline{\B}=\left(\begin{array}{ccc}
  \overline{a}_0&\overline{b}_0&0\\
  \overline{c}_1&\overline{a}_1&\overline{b}_1\\
  0&\overline{c}_2&\overline{a}_2\\
\end{array}\right)
=
 \left(\begin{array}{ccc}
  0&3&0\\
  1&0&2\\
  0&\frac{3}{2}&\frac{3}{2}\\
\end{array}\right).
$$

A second example of distance mean-regular, but not vertex-transitive, graph $\G$
is shown in Fig. \ref{fig2}.
\begin{figure}[h]
\center
  \includegraphics[scale=0.7]{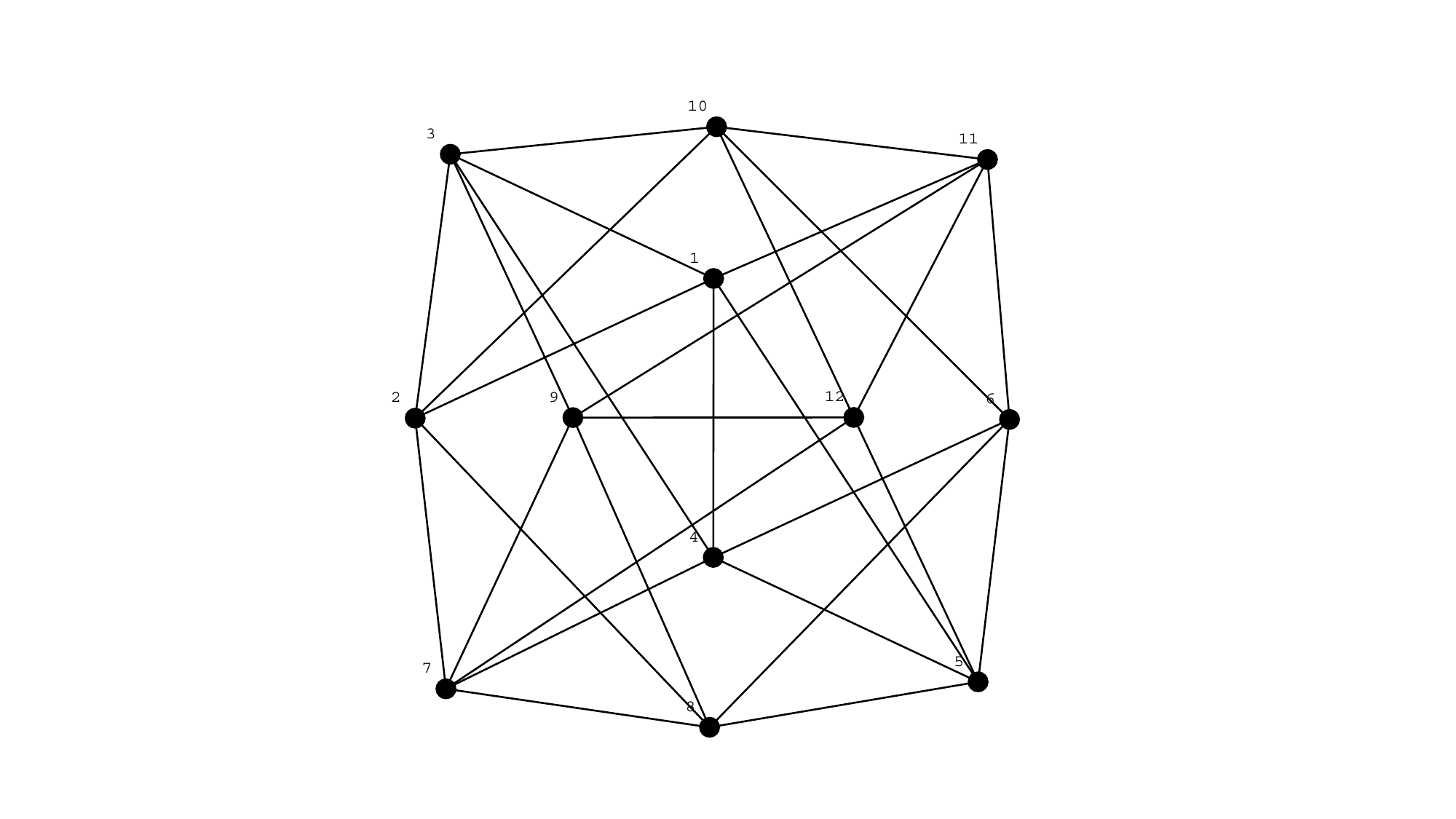}
  \caption{A non-vertex transitive distance mean-regular graph.}
  \label{fig2}
\end{figure}
Now, $\G$ has 
again diameter $D=2$, and intersection mean-matrix
$$
\overline{\B}=\left(\begin{array}{ccc}
  \overline{a}_0&\overline{b}_0&0\\
  \overline{c}_1&\overline{a}_1&\overline{b}_1\\
  0&\overline{c}_2&\overline{a}_2\\
\end{array}\right)
  =
  \left(\begin{array}{ccc}
  0&5&0\\
  1&\frac{6}{5}&\frac{14}{5}\\
  0&\frac{14}{6}&\frac{16}{6}\\
\end{array}\right)
$$
In fact, we can easily check that $\G$ is neither vertex-transitive nor distance-regular by comparing the subgraphs induced by $\G_1(1)$  (a $4$-path and a singleton), and $\G_1(2)$ (a $2$-path and a $3$-path).


\subsection{A first characterization}
From the results in Subsection \ref{interlacing}, it is clear that the intersection mean-matrix $\overline{\B}$ of a distance mean-regular graph $\G$ is, in fact, the quotient matrix of the distance partition with respect to any vertex, computed as in \eqref{quotient-matrix}.
This leads us to the following result.

\begin{proposition}
A graph $\G$ is distance mean-regular if and only if the quotient matrix $\B$, evaluated in \eqref{quotient-matrix},  with respect to the distance partition of every vertex, is the same and, in this case, $\B$ turns out to be the intersection mean-matrix $\overline{\B}$ of $\G$. Moreover, $\G$ is distance-regular if and only if the interlacing is tight.
\end{proposition}
\begin{proof}
We only need to prove the sufficiency of the second statement. If the interlacing is tight, the distance partition of every vertex is equitable. Moreover, the graph is clearly regular. Then, distance-regularity follows from the result of Godsil and Shawe-Taylor \cite{gs87} (see also \cite{f12}).
\end{proof}

\subsection{The intersection mean-parameters}
Let $\G$ have diameter $D$. For $i=0,\ldots, D$, let $\overline{\B}_i$ be the {\em proper intersection-$i$ mean-matrix} with entries $(\overline{\B}_i)_{hj}=\overline{p}_{i,j}^h$. In particular, notice that $\overline{\B}_0=\I$, and $\overline{\B}_1=\overline{\B}$. In general, $\overline{\B}_i$ can be easily computed in the following way.

\begin{lemma}
Let $\S$ and $\T$ be the matrices corresponding to a distance partition with respect to a given vertex of a distance mean-regular graph. Then, its proper intersection-$i$ mean-matrix can be computed form its distance-$i$ matrix by the formula
\begin{equation}
\label{Ai->Bi}
\overline{\B}_i=\S^{\top}\A_i \T,\qquad i=0,\ldots,D.
\end{equation}
\end{lemma}
\begin{proof}
Let $U_i=\G_i(u)$, with $k_i=|\G_i(u)|$, $i=0,\ldots,D$, be the distance partition with respect to vertex $u$. Then, the matrices $\T$ and $\S$ have entries
$$
(\T)_{wi}=\left\{
\begin{array}{ll}
1 & \mbox{if $w\in U_i$}, \\
0 & \mbox{otherwise,}
\end{array}
\right.
\qquad \mbox{and}\qquad
(\S)_{vi}=\left\{
\begin{array}{ll}
1/k_i & \mbox{if $v\in U_i$}, \\
0 & \mbox{otherwise.}
\end{array}
\right.
$$
Then, for every $h,j=0,\ldots,D$,
\begin{align*}
(\S^{\top}\A_i\T)_{hj}&=\sum_{v,w\in V}(\S^{\top})_{hv}(\A_i)_{vw}(\T)_{wj}
=\sum_{v\in \G_h(u), w\in\G_i(v)\cap \G_j(u)} \frac{1}{k_h}\\
 &=\frac{1}{k_h}\sum_{v\in \G_h(u)}|\G_i(v)\cap \G_j(u)|=\overline{p}_{ji}^h=\overline{p}_{ij}^h=(\overline{\B}_i)_{hj},
\end{align*}
where we have used the symmetric property $\overline{p}_{ji}^h=\overline{p}_{ij}^h$, which is proved later in Proposition \ref{propo-dist-mat}.
\end{proof}

\subsection{Some properties}
As it could be expected, some combinatorial properties of  distance mean-regular graphs are similar to those of distance-regular graphs. The following result, which is not exhaustive, shows some simple examples.

\begin{lemma}
\label{basic-properties}
Let $\G$ be a distance mean-regular graph with diameter $D$ and parameters
as above. Then, the following holds:
\begin{itemize}
\item[$(i)$]
 The graph $\G$ is regular with degree $k=\overline{b}_0$, with $k=\overline{a}_i+\overline{b}_i+\overline{c}_i$ for every $i=0,\ldots,D$.
\item[$(ii)$]
For every vertex $u$ and $i=0,\ldots,D-1$, $k_i(u)\overline{b}_i=\overline{c}_{i+1}k_{i+1}(u)$.
\item[$(iii)$]
$\G$ is {\em distance-degree regular}: $k_i(u)=k_i$ for every $u\in V$ and $i=0,\ldots,D$.
\end{itemize}
\end{lemma}

\begin{proof}
$(i)$ The fist statement is clear since
$$
\overline{a}_i+\overline{b}_i+\overline{c}_i=\frac{1}{|\G_i(u)|}\sum_{v\in \G_i(u)}a_i(u,v)+b_i(u,v)+c_i(u,v),
$$
and, as $a_i(u,v)+b_i(u,v)+c_i(u,v)$ is the degree $\delta(v)=k_1(v)$ for every $v\in \G$, we have $\overline{a}_i+\overline{b}_i+\overline{c}_i=\frac{1}{|\G_i(u)|}\sum_{v\in \G_i(u)}=k$.

$(ii)$ This follows by counting in two ways the edges between $\G_i(u)$ and $\G_{i+1}(u)$:
$$
|\G_i(u)|\overline{b}_i= \sum_{v\in \G_{i}(u)} |\G_1(v)\cap \G_{i+1}(u)|
= \sum_{w\in \G_{i+1}(u)} |\G_1(w)\cap \G_{i}(u)|
=\overline{c}_{i+1}|\G_{i+1}(u)|.
$$
$(iii)$ By induction using  $(ii)$ and starting from $k_0(u)=1$ (or, by $(i)$, $k_1(u)=k$) for every $u\in V$ .
\end{proof}

Also, the intersection mean-numbers often gives similar information to that provided for the corresponding parameters of distance-regular graphs. First, recall that the {\em odd-girth} of a graph is the minimum length of an odd cycle, and the {\em even-girth} is defined analogously.

\begin{lemma}
Let $\G$ be a distance mean-regular graph with intersection mean-numbers as above. Then, the following holds:
\begin{itemize}
\item[$(i)$]
If $\overline{a}_i=0$ for all $i<m$ and $\overline{a}_m\neq 0$, then $\G$ has odd-girth $2m+1$.
\item[$(ii)$]
$\G$ has no even cycles if and only if $\overline{c}_i=1$ for any $i$. In general, the even-girth of $\G$ is $2i$, where $i$ is the minimum value such that $\overline{c}_i>1$;
\end{itemize}
\end{lemma}
\begin{proof}
$(i)$ If, for every $i<m$, $\overline{a}_i=0$, for every pair of vertices $u,v$ at distance $i$, we must have
$a_i(u,v)=0$. Then $\G$ cannot contain any odd cycle of length $2i+1$. Moreover, if $\overline{a}_m\neq 0$, it must be some vertices $u,v$ such that $a_m(u,v)\ge 1$ and, hence, there is an odd cycle of length $2m+1$.
The proof of $(ii)$ is analogous.
\end{proof}

In contrast with the above, some other well-known properties of distance-regular graphs are not shared by the distance mean-regular graphs. For instance, the intersection numbers $b_i$ and $c_i$ of a distance-regular graph satisfy the monotonic properties
$\overline{b}_0\ge \overline{b}_1\ge \overline{b}_2\ge \cdots $ and
$\overline{c}_1\le \overline{c}_2\le \overline{c}_3\ge \cdots $, whereas  this is not the case for the intersection mean-parameters $\overline{b}_i$ and
$\overline{c}_i$ (see, e.g. the truncated tetrahedron studied in Subsection \ref{trunc-tetra}).

Since every distance-regular graph is also distance mean-regular, one natural question would be the following:
\begin{quest}
If the intersection mean-parameters of a distance mean-regular graph are integers, is there always a distance-regular graph with these parameters?
\end{quest}

In fact, Brouwer \cite{b15} gave a negative answer to this question by noting that the Cayley graph $\G=\Cay(\Z_{21}; \pm i, i=1,\ldots,5\}$, with diameter $D=2$, has the parameters of a nonexisting strongly regular graph. Indeed, as a distance
mean-regular, $\G$ has the integer intersection mean-matrix
$$
\overline{\B}=\left(\begin{array}{ccc}
  \overline{a}_0&\overline{b}_0&0\\
  \overline{c}_1&\overline{a}_1&\overline{b}_1\\
  0&\overline{c}_2&\overline{a}_2\\
\end{array}\right)
  =
  \left(\begin{array}{ccc}
  0&10&0\\
  1&6&3\\
  0&3&7\\
\end{array}\right).
$$
However, $\G$ has $d+1=11$ distinct eigenvalues and, hence, it cannot be distance-regular $(D\neq d)$.

%


\subsection{Orthogonal polynomials}
\label{ortho-pol}
From the proper intersection mean-matrix $\overline{\B}$ of a distance mean-regular graph $\G$ with $n$ vertices and diameter $D$,
we can construct an orthogonal sequence of polynomials by using the three-term recurrence
\begin{equation}
\label{3-term-pol}
x\overline{p}_i=\overline{b}_{i-1}\overline{p}_{i-1}+  \overline{a}_{i}\overline{p}_{i}+\overline{c}_{i+1}\overline{p}_{i+1},\qquad i=0,1,\ldots,D,
\end{equation}
initiated with $\overline{p}_0=1$ and $\overline{p}_1=x$, and where, by convention, $\overline{b}_{-1}=\overline{c}_{i+1}=0$.

According to the results in C\'amara,  F\`abrega,  Fiol, Garriga \cite{cffg09},
these polynomials, which here will be called the {\em distance mean-polynomials}, are orthogonal with respect to the scalar product
\begin{equation}
\label{star-product}
\langle f,g\rangle_{\star}=\frac{1}{n}\sum_{i=0}^g w_i f(\mu_i)g(\mu_i)
\end{equation}
where  $\mu_0>\mu_1>\cdots>\mu_{D}$ are the distinct eigenvalues of $\overline{\B}$, and their {\em pseudo-multiplicities} $w_i$ are computed with the formulas
\begin{equation}
\label{pseudo-mul}
w_i = \frac{\pi_0 \overline{p}_D(\mu_0)}{\pi_i\overline{p}_D(\mu_i)}
\end{equation}
where $\pi_i=\prod_{j\ne i}|\mu_i-\mu_j|$, $i=0,\ldots,D$.
Moreover, these polynomials are normalized in such a way that $\|\overline{p}_i\|^2_{\star}=\overline{p}_i(\lambda_0)=k_i$ for $i=0,\ldots,D$.

By evaluating these polynomials at $\A$ or $\overline{\B}$, we obtain, respectively, the so-called {\em distance mean-matrices}
\begin{equation}
\label{mean-pols->mean-matrices}
\overline{\A}_i=\overline{p}_{i}(\A),\qquad i=0,\ldots,D,
\end{equation}
and the {\em intersection mean-matrices}
\begin{equation}\label{mean-pols->intersec-mean-matrices}
\overline{\B}_i=\overline{p}_i(\overline{\B}),\qquad i=0,\ldots,D.
\end{equation}
Of course, both families of matrices satisfy a three-term recurrence like \eqref{3-term-pol}. For instance, the distance mean-matrices satisfy
\begin{equation}\label{recur-dist-mean-matrices}
\A\overline{\A}_i=\overline{b}_{i-1}\overline{\A}_{i-1}+  \overline{\A}_{i}\overline{\A}_{i}+\overline{c}_{i+1}\overline{\A}_{i+1},\qquad i=0,1,\ldots,D,
\end{equation}
starting  with $\overline{\A}_0=\I$ and $\overline{\A}_1=\A$ (by convention, $\overline{\A}_{-1}=\overline{\A}_{i+1}=\00$).
Besides, because of \eqref{mean-pols->mean-matrices} and \eqref{mean-pols->intersec-mean-matrices}, both $\overline{\A}_i$ and $\overline{\B}_i$ have constant row sum $p_i(\lambda_0)$. For instance, $\overline{\A}_i\j=\overline{p}_i(\lambda_0)\j=k_i\j$, where $\j$ is the all-1 vector.

Concerning the intersection mean-matrices, notice that $\overline{\B}_0=\overline{p}_0(\overline{\B})=\I$, with $(\I)_{hj}=\overline{p}_{0j}^h$, and $\overline{\B}_1=\overline{p}_0(\overline{\B})=\overline{\B}$, with $(\overline{\B})_{hj}=\overline{p}_{1j}^h$.
In general, in Section \ref{algebras} we will show that, under some conditions, the intersection-$i$ matrix is proper for every $i=0,\ldots,D$. That is, it satisfies \eqref{Ai->Bi} and, hence,
\begin{equation}
\label{entries-Bi}
(\overline{\B}_i)_{hj}=\overline{p}_{ij}^h.
\end{equation}
%

\subsection{An Example}
As an example, consider 
the prism $\G=C_5\times \G=K_2$ shown in Fig. \ref{C5xK2}, which is a distance
mean-regular graph with spectrum
(here and henceforth, numbers are rounded at three decimals)
$$
\spec \G=\{3, 1.618^2,  1^1, -0.382^2, -0.618^2, -2.618^2\}.
$$
\begin{figure}[t]
\begin{center}
\begin{tikzpicture}[scale=1.5]

\path (0,0) coordinate (P1);
\node  at (0,0.3) {$1$};
\path (-1,-1) coordinate (P2);
\node  at (-1.3,-1) {$2$};
\path (0,-1) coordinate (P3);
\node  at (0,-1.3) {$6$};
\path (1,-1) coordinate (P4);
\node  at (1.3,-1) {$5$};

\path (-0.4,-2) coordinate (P5);
\node  at (-0.7,-2) {$3$};
\path (0.5,-2) coordinate (P6);
\node  at (0.8,-2) {$4$};
\path (-1.5,-2) coordinate (P7);
\node  at (-1.8,-2) {$7$};
\path (1.5,-2) coordinate (P8);
\node  at (1.8,-2) {$10$};

\path (-0.65,-3) coordinate (P9);
\node  at (-0.65,-3.3) {$8$};
\path (0.65,-3) coordinate (P10);
\node  at (0.65,-3.3) {$9$};

\foreach \i in {1,...,10}
{
\draw[fill] (P\i) circle (2pt);
}

\foreach \i / \j  in {1/2,1/3,1/4,2/5,2/7,3/7,8/3,6/4,8/4,9/5,10/6,7/9,10/8,10/9,6/5}
{
 \draw (P\i) --(P\j);
}

\node  at (2.5,0) {$U_0$};
\node  at (2.5,-1) {$U_1$};
\node  at (2.5,-2) {$U_2$};
\node  at (2.5,-3) {$U_3$};

\end{tikzpicture}
\end{center}
\caption{The prism $\G=C_5 \times K_2$}
\label{C5xK2}
\end{figure}
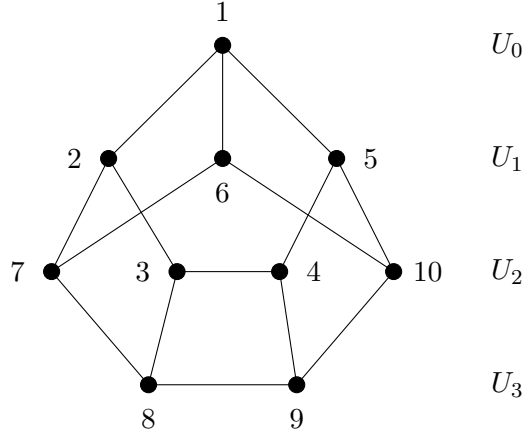
Its quotient matrix with respect to the distance partition ${\cal P}=\{1\}\cup\{2,6,5\}\cup\{7,3,4,10\}\cup\{8,9\}$ can be obtained from the matrices
$$
\T^{\top}=\left(\begin{array}{cccccccccc}
  1&0&0&0&0&0&0&0&0&0\\
  0&1&0&0&1&1&0&0&0&0\\
  0&0&1&1&0&0&1&0&0&1\\
0&0&0&0&0&0&0&1&1&0\\
\end{array}\right)
\hspace{2mm}
\begin{array}{c}
U_0\\
U_1\\
U_2\\
U_3\\
\end{array},
$$
$\D=\diag(1, 3, 4, 2)$, and
$$
\S^{\top}=\D^{-1}\T^{\top}=\left(\begin{array}{cccccccccc}
  1&0&0&0&0&0&0&0&0&0\\
  0&\frac{1}{3}&0&0&\frac{1}{3}&\frac{1}{3}&0&0&0&0\\
  0&0&\frac{1}{4}&\frac{1}{4}&0&0&\frac{1}{4}&0&0&\frac{1}{4}\\
0&0&0&0&0&0&0&\frac{1}{2}&\frac{1}{2}&0\\
\end{array}\right)
\hspace{2mm}
\begin{array}{c}
U_0\\
U_1\\
U_2\\
U_3\\
\end{array},
$$
stisfying
$\S^{\top} \T=\I$.
Then, the (proper) intersection mean-matrix is
$$
\overline{\B}= \S^{\top}\A\T=\left(\begin{array}{cccc}
  0&3&0&0\\
  1&0&2&0\\
  0&\frac{3}{2}&\frac{1}{2}&1\\
0&0&2&1\\
\end{array}\right)
=
\left(\begin{array}{cccc}
  \overline{a}_0&\overline{b}_0&0&0\\
  \overline{c}_1&\overline{a}_1&\overline{b}_1&0\\
  0&\overline{c}_2&\overline{a}_2&\overline{b}_2\\
0&0&\overline{c}_3&\overline{a}_3\\
\end{array}\right),
$$
with eigenvalues
$$
\mu_1=3,\quad \mu_2=1.402,\quad \mu_3=-0.433,\quad \mu_4=-2.469,
$$
which interlace the eigenvalues of $\A$.
The distance mean-polynomials, and their values at $\lambda_0=3$, are
\begin{equation*}
\begin{array}{ll}
\label{}
\overline{p}_0(x) = 1,                           & \overline{p}_0(3)=1,\\
\overline{p}_1(x) = x,                           & \overline{p}_1(3)=3,\\
\overline{p}_2(x) = \frac{1}{3} (2x^2-6),        &\overline{p}_2(3)=4,\\
\overline{p}_3(x) =  \frac{1}{6}(2x^3-x^2-12x+3),& \overline{p}_4(3)=2.
\end{array}
\end{equation*}
From $p_3$ we then compute the pseudo-multiplicities by using \eqref{pseudo-mul}, which are
$w_0=1$, $w_1= 3.085$, $w_2=3.575$, and $w_3= 2.340$ (notice that, as required,  they sum up to $n=10$).
Moreover, besides $\overline{p}_0(\overline{\B})=\I$ and $\overline{p}_1(\overline{\B})=\overline{\B}$, we have the other proper intersection mean-matrices:
\begin{align*}
\overline{\B}_2 & =\S^{\top}\A_2\T=\overline{p}_2(\overline{\B})=
\left(\begin{array}{cccc}
  0&0&4&0\\
  0&2&\frac{2}{3}&\frac{4}{3}\\
  1&\frac{1}{2}&\frac{3}{2}&1\\
  0&2&2&0\\
\end{array}\right), \\
\overline{\B}_3 &=\S^{\top}\A_3\T=\overline{p}_3(\overline{\B})=
\left(\begin{array}{cccc}
  0&0&0&2\\
  0&0&\frac{4}{3}&\frac{2}{3}\\
  0&1&1&0\\
1&1&0&0\\
\end{array}\right).
\end{align*}

\section{Characterizations}
In this section we give several (combinatorial and algebraic) characterizations of distance mean-regular graphs.

\subsection{The numbers of edges}
Distance mean-regular graphs have to do with the invariance of the number of edges whose endvertices are  at a given  distance from each vertex $u$. More precisely, for a graph $\G=(V,E)$ with diameter $D$, a vertex $u$, and integers $i,j=0,\ldots,D$, let us consider the following parameters:
\begin{align*}
\omega_{ij}(u) & = |\{vw\in E : \dist(u,v\}=i, \dist(u,w)=j\}|.
\end{align*}
Notice that $\omega_{ij}(u)=0$ if $|i-j|>1$. Besides, some trivial values are $\omega_{00}(u)=0$ (since there are no loops), and $\omega_{01}=\delta(u)$, the degree of $u$.
If these numbers do not depend on $u$, we say that they are {\em well defined}, and represent them as $\omega_{ij}$.

\begin{lemma}
\label{charac1}
A graph $\G$ with diameter $D$ is distance mean-regular if and only if the numbers
$\omega_{ij}$ are well defined for every $i,j=0,\ldots,D$.
\end{lemma}
\begin{proof}
If $\G$ is distance mean-regular, then  $\omega_{ij}$ is well defined since, by Lemma \ref{basic-properties}, $\omega_{ii}=\overline{a}_i k_i/2$ and $\omega_{i,i+1}=k_i \overline{b}_i=k_{i+1}\overline{c}_{i+1}$.
Conversely, if the $\omega_{ij}$'s are well defined, $\G$ is regular with degree $k=\omega_{01}$. Moreover, for any vertex $u$,
$$
k_i(u)=\frac{1}{k_i}\sum_{v\in \G_i(u)}\delta(v)=\frac{1}{k_i}(\omega_{i-1,i}+2\omega_{i,i}+\omega_{i,i+1}),
$$
so that $k_i$ is well defined. Then, $\overline{a}_i=\frac{2}{k_i}\omega_{i,i}$,
$\overline{b}_i=\frac{1}{k_i}\omega_{i,i+1}$, and $\overline{c}_i=\frac{1}{k_i}\omega_{i-1,i}$ are also well defined, and $\G$ is distance mean-regular.
\end{proof}

Although  every distance mean-regular is super-regular (Lemma \ref{basic-properties}$(iii)$), the converse is not true.
A counterexample is, for instance, the graph of Fig. \ref{figxyz} on $n=12$ vertices and diameter $D=2$, which is super-regular with $k_1=5$ and $k_2=6$, but not distance mean-regular.

 \begin{figure}[h]
\center
  \includegraphics[scale=0.8]{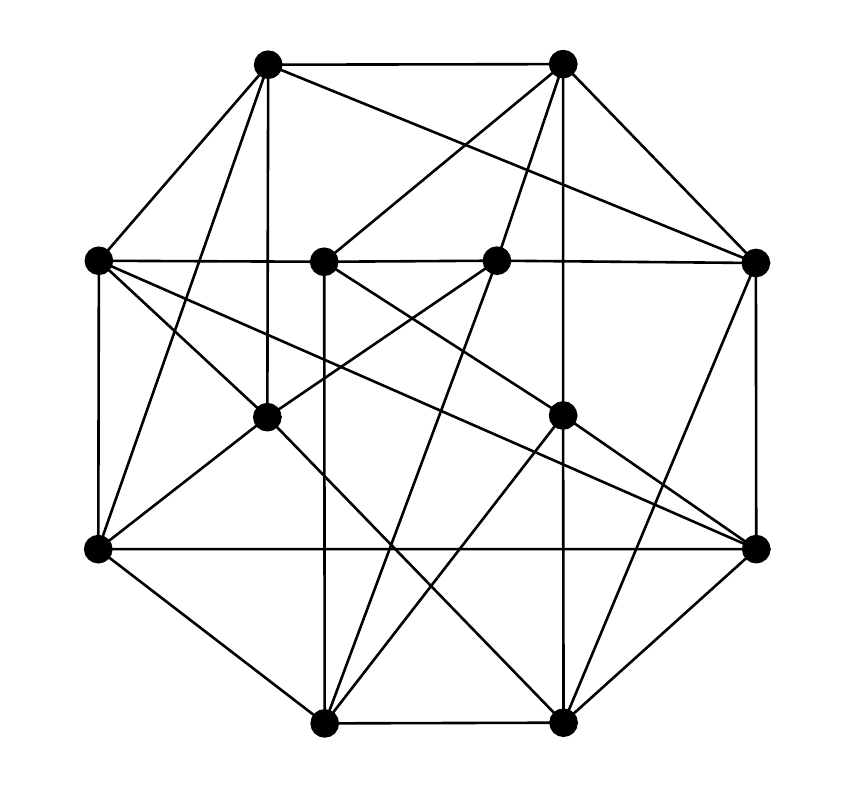}
  \caption{A super-regular, but not distance mean-regular, graph.}
  \label{figxyz}
\end{figure}


Hovewer, we have the following characterization of those super-regular graphs which are distance mean-regular:

\begin{proposition}
Let $\G$ be a super-regular graph. Then, $\G$ is distance mean-regular if and only if $\omega_{ii}$ is well defined for every $i=0,\ldots,D$.
\end{proposition}
\begin{proof}
Necessity follows from Lemma \ref{charac1}.
To prove sufficiency, let $\G$ be a super-regular graph with well-defined $\omega_{ii}$  for every $i=0,\ldots,D$. Then,  $\overline{a}_i=2\omega_{ii}/k_i$ is well defined.
Let us now proceed by induction on $i$. Since $\G$ is regular,  $\overline{b}_0=k$ is well defined. Now suppose that all the intersection mean-numbers are well defined until $\overline{b}_i$. Then, from Lemma \ref{basic-properties}$(ii)$, $\overline{c}_{i+1}=\overline{b}_i\frac{k_i}{k_{i+1}}$ is well defined, and, from Lemma \ref{basic-properties}$(i)$,
$\overline{b}_{i+1}=k-\overline{a}_{i+1}-\overline{c}_{i+1}$ also is. This complete the induction step.
\end{proof}

\subsection{The numbers of triples}
Other parameters that characterize distance mean-regularity are the numbers of triples of vertices at some given distances between them.
More precisely, given a graph $\G$ with diameter $D$, three integers $h,i,j=0,\ldots,D$, and a fixed vertex $u$, let $t_{hij}(u)$ the number of triples $u,v,w\in V$ such that $\dist(u,v)=h$, $\dist(u,w)=i$, and $\dist(v,w)=j$. Note that these numbers can be computed in two ways
$$
t_{hij}(u)=\sum_{v\in \G_h(u)}|\G_j(v)\cap \G_i(u) |=\sum_{w\in \G_i(u)}|\G_j(w)\cap \G_h(u)|.
$$
Thus, if $\G$ is distance mean-regular $t_{hij}$ is {\em well defined} (that is, it does not depend on $u$) since
\begin{equation}
\label{num-triples}
t_{hij}=k_h\overline{p}_{ji}^h=k_i\overline{p}_{jh}^i.
\end{equation}
In particular, with $j=1$ and $i=h+1$, the equation above gives $k_h\overline{p}_{1,h+1}^h=k_{h+1}\overline{p}_{1h}^{h+1}$, which corresponds to  Lemma \ref{basic-properties}$(ii)$-$(iii)$.

Conversely, if the numbers $t_{hij}$ are well defined for $h,i,j=0,\ldots,D$, then $\G$ is super regular with $k_h=t_{hh0}$, and \eqref{num-triples} yields that the intersection numbers $\overline{p}_{ji}^h=t_{hij}/k_h$ are also well defined. Hence, we have proved the following characterization.

\begin{proposition}
A graph $\G$ with diameter $D$ is distance mean-regular if and only if the numbers of triples $t_{hij}$ are well defined for $h,i,j=0,\ldots,D$.
\hfill $\square$
\end{proposition}

\subsection{The distance matrices}
Within the vector space of real symmetric $n\times n$ matrices, here we use the standard scalar product
$$
\langle \M,\N\rangle=\frac{1}{n}\tr (\M\N)= \frac{1}{n}\som (\M\circ \N),
$$ where `$\circ$' stands for the entrywise or Hadamard product, and $\som(\cdot)$ denotes the sum of the entries of the
corresponding matrix. 

In terms of the distance matrices of a graph, we have the following characterization of distance mean-regular graphs.
\begin{proposition}
\label{propo-dist-mat}
A graph $\G$ with diameter $D$ and distance matrices $\A_0,\A_1,\ldots,\A_D$ is distance mean-regular if and only if, for any $h,i,j=0,\ldots,D$,
the matrix $\A_i\A_j\circ \A_h$ has left and right eigenvectors $\j^{\top}$ and $\j$. Then, in both cases, the corresponding eigenvalues are equal to $k_h\theta$, where  $k_h=\|\A_h\|^2$ is the (common) number of vertices at distance $h$ from any vertex, and
\begin{equation}
\label{commut}
\theta=\overline{p}_{ij}^h=\overline{p}_{ji}^h=\frac{\langle \A_i\A_j, \A_h\rangle}{\|\A_h\|^2}.
\end{equation}
\end{proposition}
\begin{proof}
Let $u,v$ be two vertices at distance $h$. Then,
$$
(\A_i\A_j)_{uv}=\sum_{w\in V}(\A_i)_{uw}(\A_j)_{wv}=|\G_i(u)\cap \G_j(v)|=p_{ij}^h(u,v).
$$
Thus, $\A_i\A_j\circ \A_h$ has right eigenvalue $\j$, if and only if all its row sums have constant value
$\sum_{v\in\G_h(u)}p_{ij}^h(u,v)$.
In particular, when $h=i$ and $j=0$, the above means that the number of vertices at distance $h$ from every vertex $u$, $|\G_h(u)|$, has constant value $k_h$. Consequently,
$$
\overline{p}_{ij}^h=\frac{1}{k_h}\sum_{v\in\G_h(u)}(\A_i\A_j)_{uv}
=\frac{1}{k_h}\som(\A_i\A_j\circ \A_k)
=\frac{\langle \A_i\A_j, \A_h\rangle}{\|\A_h\|^2}.
$$
Finally, reasoning analogously, we get that $\A_h$ has left eigenvalue $\j^{\top}$ if and only if its column sums have constant value $\sum_{u\in\G_h(v)}p_{ij}^h(v,u)=k_h \overline{p}_{ji}$. Then, by adding in two ways all the entries of $\A_i\A_j\circ \A_h$ we conclude that $\overline{p}_{ij}^h=\overline{p}_{ji}^h$.
\end{proof}


\section{Algebras}
\label{algebras}
In this section we show that, as in the theory of distance regular-graphs,
distance mean-regular graphs have associated some matrix (and polynomial) algebras,
from which we can retrieve some of their main parameters.

\subsection{Basic concepts}
Let us first recall some basic concepts about associative algebras (see e.g. \cite{hggk04}).
Let ${\cal A}$ be a finite-dimensional (associative) algebra over a field $K$. Then, its bilinear multiplication from ${\cal A} \times {\cal A}$ to ${\cal A}$, denoted by `$\star$', is completely determined by the multiplication of basis elements of ${\cal A}$. Conversely, once a basis for ${\cal A}$ has been chosen, the products of basis elements can be set arbitrarily, and then extended in a unique way to a bilinear operator on ${\cal A}$, so giving rise to a (not necessarily associative)  algebra.
Thus, ${\cal A}$ can be specified, up to isomorphism, by giving its dimension (say $d$), and specifying $d^3$ {\em structure coefficients} $c_{h,i,j}$, which are scalars and determine the multiplication in ${\cal A}$ via the following rule:
$$
\e_{i}\star \e_{j} = \sum_{h=1}^d c_{h,i,j} \e_{h},\qquad h,i,j=0,1,\ldots d,
$$
where $\e_1,...,\e_d$ form a basis of ${\cal A}$.

A {\em representation} of an associative algebra ${\cal A}$ is a vector space $V$ equipped with a linear mapping $\varrho: {\cal A}\rightarrow \End V$ preserving the product and the unit.
A representation is {\em faithful} when $\varrho$ is injective. Then, distinct elements of $\x\in{\cal A}$ are represented by distinct elements  $\varrho(\x)\in \End V$.

\subsection{Algebras from distance mean-regular graphs}
Let $\G$ be a distance mean-regular graph with diameter $D$, adjacency matrix $\A$, and $d+1$ distinct eigenvalues.
Then, we can consider the following vector spaces over $\Re$:

\begin{itemize}
\item
${\cal A}=\spaan(\I,\A,\A^2,\ldots,\A^d)$;
\item
${\cal D}=\spaan(\I,\A,\A_2,\ldots,\A_D)$;
\item
$\overline{{\cal D}}=\spaan(\I,\A,\overline{\A}_2,\ldots,\overline{\A}_D)$;
\item
$\overline{{\cal B}}=\spaan(\I,\overline{\B},\overline{\B}_2,\ldots,\overline{\B}_D)$.
\end{itemize}

As it is well known, ${\cal A}$ is an algebra with the ordinary product of matrices, and it is called the {\em adjacency algebra} of $\G$.
Moreover $\G$ is distance-regular if and only if ${\cal D}={\cal A}$, which implies $D=d$ (see e.g. \cite{bcn89,b93}). In this case ${\cal A}$ is the so-called {\em Bose-Mesner algebra} of $\G$, where multiplication on the basis $\I,\A,\ldots,\A_d$ is again the usual product of matrices since
\begin{equation}
\label{mult-mean-matrices}
\A_i\A_j=\sum_{h=0}^d p_{ij}^h\A_h,\qquad h,i,j=0,1,\ldots,d.
\end{equation}

To obtain an algebra from  $\overline{{\cal D}}$,  we define the {\em star  product} `$\star$' in the following way  (recall that $\langle \M,\N\rangle=\frac{1}{n}\tr(\M\N)$).
\begin{equation}
\label{mult-dist-matrices}
\A_i\star \A_j= \sum_{h=0}^D \overline{p}_{ij}^h \A_h,\qquad h,i,j=0,1,\ldots,D.
\end{equation}
Note that, since the intersection mean-numbers $\overline{p}_{ij}^h$ are Fourier coefficients (see \eqref{commut}), the product $\A_i\star \A_j$ is just the orthogonal projection of $\A_i\A_j$ on $\overline{{\cal D}}$.
Then, we can enunciate the main result of this section.

\begin{theorem}
\label{theo-algebras}
Let $\G$ be a distance mean-regular graph with diameter $D$, adjacency matrix $\A$, and $d+1$ distinct eigenvalues.
Then the following holds.
\begin{itemize}
\item[$(i)$]
The vector space $\overline{{\cal D}}$ is a subalgebra of ${\cal A}$ with the ordinary product of matrices, and
$$
\dim\overline{{\cal D}}=D\le d=\dim{\cal A}.
$$
\item[$(ii)$]
The vector space ${\cal D}$ is a commutative (but not necessarily associative) algebra with the star product `$\star$'.
\item[$(iii)$]
The algebra $\overline{{\cal D}}$ is isomorphic to $\overline{{\cal B}}$ via $\phi:\overline{{\cal D}} \stackrel{\sim}{\longrightarrow} \overline{{\cal B}}$, where
$$
\phi(\overline{\A}_i)=\overline{\B}_i,\quad i=0,1,\ldots,D.
$$
\item[$(iv)$]
If the algebra $({\cal D},\star)$ is associative, then it is faithfully represented by $\overline{{\cal B}}$.
\end{itemize}
\end{theorem}

\begin{proof}
$(i)$ is a direct consequence of \eqref{mean-pols->mean-matrices}, \eqref{mult-mean-matrices}, and the well-known fact that $D\le d$.

$(ii)$ It can be easily checked that, the `star product' of two matrices $\X=\sum_i{x_i \A_i}$ and $\Y=\sum_j{y_j \A_j}$ is a linear combination in terms of the basis as:
$$
\X\star \Y=\sum_h\left(\sum_{i,j}x_iy_j \overline{p}_{ij}^h\right)\A_h.
$$
Thus, such an operation is bilinear, and the vector space generated by ${\cal D}$ is closed under it.
Moreover, the dimension of this algebra can not be smaller than $D$, because the distance matrices $\A_i$ are clearly independent.
Finally, the commutative property follows from the equalities $\overline{p}_{ij}^h=\overline{p}_{ji}^h$ for $h,i,j=0,\ldots,D$.

$(iii)$ This a consequence of the fact that, both $\overline{{\cal D}}$ and $\overline{{\cal B}}$, are isomorphic to the algebra $\Re_D[x]$ with basis $\{\overline{p}_0,\overline{p}_1,\ldots,\overline{p}_D\}$.

$(iv)$ The equation \eqref{mult-dist-matrices} indicates that left-multiplication `$\star$' by $\A_i$ can be seen as a linear mapping $\phi_i$ of ${\cal D}$ with respect to the basis $\I,\A,\ldots,\A_D$.  Moreover, $\phi_i$ is fully represented by the matrix $\overline{\B}_i^{\top}$. Then, the result follows since $({\cal D},\star)$ is commutative,
\end{proof}

Notice that $(iv)$ is similar to the result for distance-regular graphs given by Biggs in \cite[Prop. 21.1]{b93}. Some interesting consequences of this result are in the following proposition where the associativity condition on $\overline{\D}$ has been translated into a commutativity requirement in $\overline{\B}$. This is because, from \eqref{mult-dist-matrices}, \eqref{entries-Bi}, and \eqref{commut}, it can be checked that, for any $i,j,k=0,\ldots,D$,
$$
(\A_i\star \A_j)\star \A_k = \A_i\star (\A_j\star \A_k)\quad \iff \quad
\sum_{\ell=0}^D (\overline{\B}_k\overline{\B}_i)_{\ell j}\A_{\ell}
=\sum_{\ell=0}^D (\overline{\B}_i\overline{\B}_k)_{\ell j}\A_{\ell}.
$$
In fact, it can be shown that the commutativity of the $\overline{\B}_i's$ is equivalent to that of the $\A_i's$ (both with respect to the usual product of matrices).

\begin{proposition}
\label{propo-algebra}
Let $\G$ be a distance mean-regular graph with diameter $D$, and proper intersection mean-matrices $\overline{\B}_i$, $i=0,\ldots,D$, that commute with each other. Then, the following holds:
\begin{itemize}
\item[$(i)$]
The matrices $\overline{\B}_i$ satisfy
\begin{equation}\label{BiBj}
\overline{\B}_i\overline{\B}_j=\sum_{h=0}^D \overline{p}_{ij}^h \overline{\B}_h,\qquad h,i,j=0,\ldots,D
\end{equation}
and, in particular,
\begin{equation}\label{BBi}
\overline{\B}\overline{\B}_i= \overline{b}_{i-1}\overline{\B}_{i-1}+ \overline{a}_{i}\overline{\B}_{i}+ \overline{c}_{i+1}\overline{\B}_{i+1}, \qquad i=0,\ldots,D.
\end{equation}
\item[$(ii)$]
The matrix $\overline{\B}_i$ is the distance mean-polynomial of degree $i$ at $\overline{\B}$:
$$
\overline{\B}_i=\overline{p}_i(\overline{\B}),\qquad i=0,\ldots,D.
$$
\item[$(iii)$]
Every parameter $\overline{p}_{ij}^h$ is well determined by the parameters $\overline{a}_{i}$, $\overline{b}_{i}$ and $\overline{c}_{i}$.
\end{itemize}
\end{proposition}
\begin{proof}
All the results are sustained by the fact, under the hypotheses, $\overline{{\cal B}}$ is a faithful representation of $({\cal D},\star)$. Indeed, the us check that the (linear) mapping $\Psi: {\cal D} \rightarrow \overline{{\cal B}}$ defined by $\Psi(\A_i)=\overline{B}_i$, for $i=0,\ldots, D$ is an algebra isomorphism. First, using \eqref{mult-dist-matrices}, \eqref{entries-Bi}, and \eqref{commut},
\begin{align*}
\Psi(\A_i\star \A_j)&=\Psi\left(\sum_{h=0}^D \overline{p}_{ij}^h \A_h\right)
=\sum_{h=0}^D (\overline{\B}_i)_{hj}\overline{\B}_h
\end{align*}
so that, for any $r,s=0,\ldots,D$,
$$
(\Psi(\A_i\star \A_j))_{rs}=\sum_{h=0}^D (\overline{\B}_i)_{hj}(\overline{\B}_h)_{rs}
=\sum_{h=0}^D (\overline{\B}_i)_{hj}(\overline{\B}_s)_{rh}=(\overline{\B}_s\overline{\B}_i)_{rj}.
$$
Moreover,
$$
(\Psi(\A_i)\Psi(\A_j))_{rs}=(\overline{\B}_i\overline{\B}_j)_{rs}
=\sum_{h=0}^D (\overline{\B}_i)_{rh}(\overline{\B}_j)_{hs}
=\sum_{h=0}^D (\overline{\B}_i)_{rh}(\overline{\B}_s)_{hj}=(\overline{\B}_i\overline{\B}_s)_{rj}.
$$
Thus, $\Psi(\A_i\star \A_j)=\Psi(\A_i)\Psi(\A_j)$, as claimed.

Taking the above into mind, $(i)$ and $(ii)$ are direct consequences of
\eqref{mult-dist-matrices} and the results of Subsection \ref{ortho-pol}.

$(iii)$ By the same results, the parameters $\overline{a}_{i}$, $\overline{b}_{i}$ and $\overline{c}_{i}$ determine the distance
mean-polynomials $\overline{p}_0,\overline{p}_1,\ldots,\overline{p}_D$ which, according to $(ii)$, give the intersection mean-matrices $\overline{\B}_i$ with entries $\overline{p}_{ij}^h$.
Alternatively, by \eqref{BiBj} and $(ii)$, we have
$$
\overline{p}_i\overline{p}_j=\sum_{h=0}^D \overline{p}_{ij}^h \overline{p}_h,\qquad h,i,j=0,\ldots,D.
$$
Thus, $\overline{p}_{ij}^h$ is just the Fourier coefficient of $\overline{p}_i\overline{p}_j$, with respect to the scalar product in \eqref{star-product}, in terms of the basis $\{\overline{p}_h:h=0,\ldots,D\}$:
$$
\overline{p}_{ij}^h=\frac{\langle\overline{p}_i\overline{p}_j,
\overline{p}_h\rangle_{\star}}{\|\overline{p}_h\|_{\star}^2}.
$$
This completes the proof.
\end{proof}

Notice that the equalities $(\overline{\B}_i\overline{\B}_j)_{rs}=(\overline{\B}_j\overline{\B}_i)_{rs}$ for every $i,j,r,s$, which can be written as
$$
\sum_{h=0}^D \overline{p}_{sh}^r\overline{p}_{ij}^h
=\sum_{h=0}^D \overline{p}_{ih}^r\overline{p}_{sj}^h,
$$
are like the ones satisfied by the parameters $p_{ij}^h$ of an association scheme with $D=d$ classes (or a distance-regular graph with diameter $D$);
see e.g. Brouwer, Cohen and Neumaier \cite[Lemma 2.1.1(vi)]{bcn89}.

\subsection{The truncated tetrahedron}
\label{trunc-tetra}
We end this section with an example showing that the conditions of Theorem \ref{theo-algebras} (or Proposition \ref{propo-algebra}) not always hold.
The {\em truncated tetrahedron} $\G= K_4[\triangle]$ shown in Fig. \ref{trunc-tetra}, is a vertex-transitive (and Cayley) graph with diameter $D=3$, and spectrum
$$
\spec \G=\{3, 2^3,  0^2, -1^3,-2^3\}.
$$

\begin{figure}[t]
\begin{center}
 \includegraphics[scale=0.6]{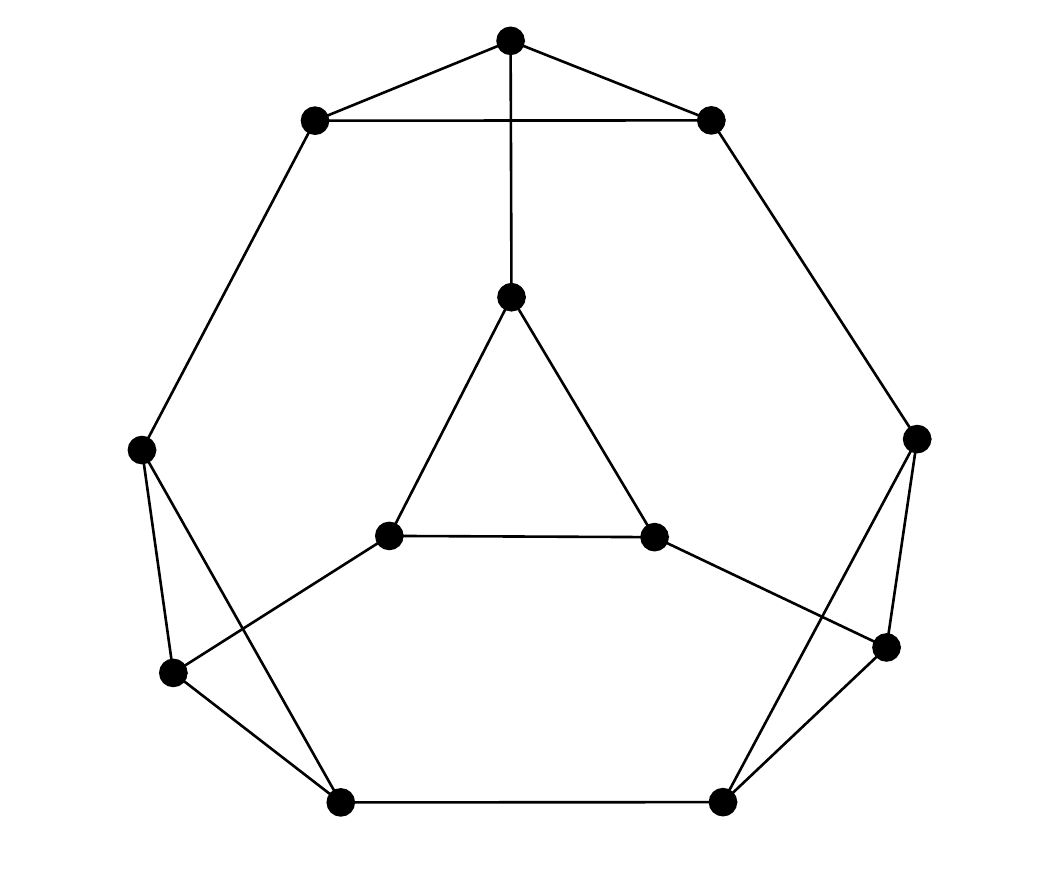}
\end{center}
\caption{The truncated tetrahedron}
\label{trunc-tetra}
\end{figure}
Thus, it is a distance mean-regular graph with proper intersection mean-matrices
$$
\overline{\B}=\left(\begin{array}{cccc}
  0&3&0&0\\
  1&\frac{2}{3}&\frac{4}{3}&0\\
  0&1&\frac{1}{2}&\frac{3}{2}\\
0&0&\frac{3}{2}&\frac{3}{2}\\
\end{array}\right),\quad \overline{\B}_2=
\left(\begin{array}{cccc}
  0&0&4&0\\
  0&\frac{4}{3}&\frac{2}{3}&2\\
  1&\frac{1}{2}&1&\frac{3}{2}\\
  0&\frac{3}{2}&\frac{3}{2}&1\\
\end{array}\right), \quad \mbox{and}\quad
\overline{\B}_3=
\left(\begin{array}{cccc}
  0&0&0&4\\
  0&0&2&2\\
  0&\frac{3}{2}&\frac{3}{2}&1\\
1&\frac{3}{2}&1&\frac{1}{2}\\
\end{array}\right).
$$
Note that, since $\overline{b}_1=\frac{4}{3}<\frac{3}{2}=\overline{b}_2$, the monotonic property $\overline{b}_0\ge \overline{b}_1\ge \overline{b}_2$ does not hold.

However, as these matrices do not commute, they are not all equal to the intersection mean-matrices obtained from the distance mean-polynomials \eqref{mean-pols->intersec-mean-matrices}, which turn out to be:
$$
\overline{p}_1(\overline{\B})=\overline{\B},\quad
\overline{p}_2(\overline{\B})=
\left(\begin{array}{cccc}
  0&0&4&0\\
  0&\frac{4}{3}&\frac{2}{3}&2\\
  1&\frac{1}{2}&\frac{1}{2}&2\\
  0&\frac{3}{2}&2&\frac{1}{2}\\
\end{array}\right), \quad \mbox{and}\quad
\overline{p}_3(\overline{\B})=
\left(\begin{array}{cccc}
  0&0&0&4\\
  0&0&2&2\\
  0&\frac{3}{2}&2&\frac{1}{2}\\
1&\frac{3}{2}&\frac{1}{2}&1\\
\end{array}\right).
$$

\noindent{\large \bf Acknowledgments.} The author acknowledges the useful comments and suggestions of A.E. Brouwer.
This research is supported by the
{\em Ministerio de Econom\'{\i}a y Competividad} (Spain) and the {\em European Regional Development Fund} under project MTM2011-28800-C02-01, and the {\em Catalan Research
Council} under project \linebreak 2009SGR1387.

\end{document}